\documentclass[12pt]{amsart}
\usepackage{epsfig}
\usepackage{mathtools}
\usepackage[subnum]{cases}
\usepackage{enumerate}
\usepackage{bbm,bm}
\usepackage{amsmath,amssymb,mathrsfs,amsthm}
\usepackage{color}
\usepackage{verbatim}
\usepackage{graphicx}
\usepackage{ctable}
\graphicspath{ {./images/} }
\usepackage{hyperref}
\definecolor{amethyst}{rgb}{0.6, 0.4, 0.8}
\definecolor{blue-violet}{rgb}{0.54, 0.17, 0.89}
\definecolor{electricviolet}{rgb}{0.56, 0.0, 1.0}

\usepackage[hang]{footmisc} %flushes footnotes. to flush left use option ``[flushmargin]"
\newtheorem{theorem}{Theorem}%[section]

\newtheorem{lemma}[theorem]{Lemma}

\newtheorem{prop}[theorem]{Proposition}
\newtheorem{remark}[theorem]{Remark}

\newtheorem*{property*}{Property}

\newcommand{\R}{\mathbb{R}} 
\newcommand{\Sp}{\mathbb{S}^{d-1}}
\newcommand{\E}{\mathbb{E}}
\newcommand{\PP}{\mathbb{P}}

\DeclareMathOperator{\var}{Var}

\makeatletter
\def\hlinewd#1{%
\noalign{\ifnum0=`}\fi\hrule \@height #1 \futurelet
\reserved@a\@xhline}
\makeatother

\begin{document}
\setcounter{footnote}{0}

\title[CLT for Random Partial Sphere Coverings]{Central Limit Theorem for Random Partial\\ Sphere Coverings in High Dimensions}

%feel free to change the title, formatting, notation, etc.

    \author{Steven Hoehner and Christoph Th\"ale}
    \date{\today}

	\subjclass[2020]{Primary: 60D05; Secondary 60F05}
    \keywords{Central limit theorem, high-dimensional probability, random covering, rate of convergence, sphere covering, stochastic geometry}	

\begin{abstract}
We study a random partial covering model on the \((d-1)\)-dimensional unit sphere, where \(N\) spherical caps are placed independently and uniformly at random, each covering a surface fraction of \(1/N\). This model provides a continuous geometric analogue of the classical balls-into-bins problem. We establish a Central Limit Theorem for the volume of the resulting random partial covering, showing that its fluctuations are asymptotically Gaussian. Moreover, we obtain a quantitative bound on the rate of convergence in the Kolmogorov distance. Our results hold both in fixed dimension and in a high-dimensional regime where the dimension grows at most logarithmically with \(N\).
\end{abstract}

\maketitle

\section{Introduction and main result}

Given $N$ caps on the $(d-1)$-dimensional unit sphere $\Sp$, each covering a proportion $1/N$ of the sphere, what is the maximum area they can cover? This version of the partial sphere covering problem has recently been investigated in \cite{HK-DCG,HK-2026}. In the high-dimensional regime, this question is  closely connected to the classical problem of estimating the covering density of $\R^d$ in the spirit of Erd\H{o}s, Few, and Rogers \cite{EFR}, as well as to the optimality of random polytopes in the approximation of convex bodies, see the remarks in \cite{HK-DCG,HK-2026} for more background.

Let $\sigma=\sigma_{d-1}$ denote the normalized surface area measure on $\Sp$, such that $\sigma(\Sp)=1$. For $N\in\mathbb{N}$ and $x\in\Sp$, let $C_N(x)$ be the geodesic spherical cap centered at $x$ with $\sigma(C_N(x))=1/N$. The optimal partial covering is defined by
\begin{equation}
    V_N^*(d) := \max_{x_1,\ldots,x_N\in\Sp}\sigma\left(\bigcup_{i=1}^N C_N(x_i)\right).
\end{equation}

For each $N$, let $X_1,\ldots,X_N$ be independent random points on $\Sp$ sampled according to $\sigma$, and define the random partial covering as
\[
    V_N := \sigma\left(\bigcup_{i=1}^N C_N(X_i)\right),
\]
see Figure \ref{fig:Simulation}. In \cite{HK-DCG}, this random construction was introduced as a continuous, geometric analogue of the classical balls-into-bins problem from probability and theoretical computer science. To see the connection, imagine throwing $N$ balls uniformly at random into $N$ discrete bins. As $N$ grows, the proportion of occupied bins approaches $1-e^{-1}$. In our setting, the continuous sphere takes the role of the bins, and the $N$ spherical caps of measure $1/N$ act as the balls. Following this intuition, it was shown that the expected covered proportion of the sphere is
\begin{equation}\label{eq:expectation}
    \E[V_N] = 1-e^{-1}+O(N^{-1}),\qquad\text{as }N\to\infty,
\end{equation}
which turns out to be independent of the ambient dimension $d$. The authors in \cite{HK-DCG} also proved variance and concentration inequalities for \(V_N\), obtaining in particular a lower variance bound of the form
\begin{equation}\label{eq:VarLowBd}
\var(V_N)\geq e^{-C_1}c_1^{d-1}N^{-1}    
\end{equation}
whenever \(N\geq C^d\), where \(c_1\in(0,1)\) and \(C,C_1>0\) are absolute constants, see Proposition~\ref{thm:HK-variance} below. Since random configurations often perform well in high-dimensional geometry, it is natural to ask whether these random partial coverings reach the optimal deterministic bounds. Recent work in \cite{HK-2026} showed that, assuming the partial coverings are antipodal, the random construction does indeed approach the optimal bound in the high-dimensional regime. Specifically, the maximum covered proportion is $1-e^{-1}$, meaning the random partial covering achieves the correct first-order asymptotic behavior.

\begin{figure}[t]
    \centering
    \includegraphics[width=0.99\linewidth]{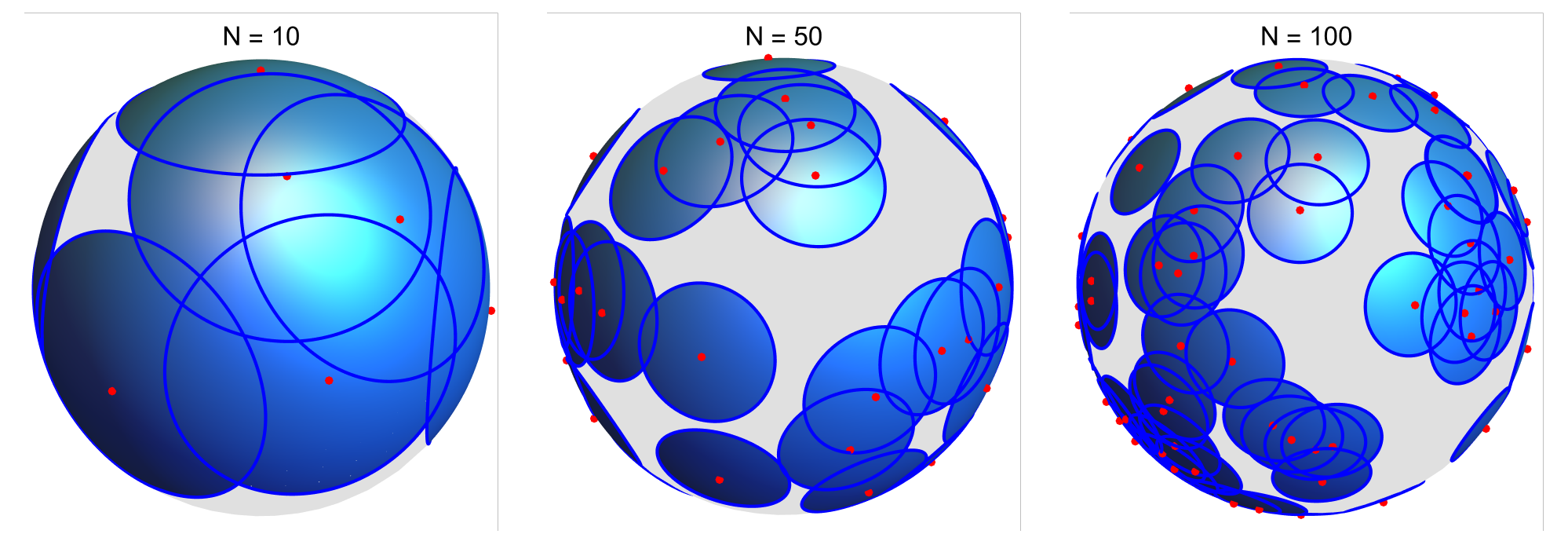}
    \caption{Simulations of random partial sphere coverings for $d=3$ and different values of $N$.}
    \label{fig:Simulation}
\end{figure}

Building on this, the present paper shows that the random partial covering is not only optimal in the mean, but that its fluctuations at the natural variance scale are asymptotically Gaussian. The main contribution of this paper is a Central Limit Theorem for random partial sphere coverings, which also provides a quantitative bound on the rate of convergence.  We measure this rate using the Kolmogorov distance between random variables \(X\) and \(Y\), defined as
\[
    d_K(X,Y) := \sup_{t\in\mathbb R}\big|\mathbb{P}(X\le t)-\mathbb{P}(Y\le t)\big|.
\]
Our result applies not only in fixed dimension, but also in a high-dimensional regime where the dimension is allowed to grow at most logarithmically with \(N\). Throughout this paper, a standard Gaussian random variable is denoted by \(\mathcal{N}(0,1)\) and convergence in distribution is indicated by the notation $\stackrel{d}{\longrightarrow}$.

\begin{theorem}\label{thm:mainThm}
    Let \(\alpha\in\bigl(0,(2\log(2/c_1))^{-1}\bigr)\) with $c_1$ as in \eqref{eq:VarLowBd}, and let \(d=d(N)\geq 2\) be a sequence of dimensions such that $d(N)\leq \alpha\log N$ for all sufficiently large \(N\). For each \(N\), let \(X_1,\ldots,X_N\) be independent random points on \(\Sp\) with distribution \(\sigma\), and define
    \[
    V_N:=\sigma\left(\bigcup_{i=1}^N C_N(X_i)\right),\qquad W_N:=\frac{V_N-\E[V_N]}{\sqrt{\var(V_N)}}.
    \]
    Then there exists an absolute constant \(C>0\), such that
    \[
    d_K(W_N,\mathcal N(0,1))\leq C\,N^{\alpha\log(2/c_1)-1/2}
    \]
    for all sufficiently large \(N\).
    In particular,
    \[
    W_N\stackrel{d}{\longrightarrow}\mathcal N(0,1)
    \qquad\text{as }N\to\infty.
    \]
\end{theorem}

\begin{remark}\label{rem:FixedDimRate}
The regime where $d$ is fixed is clearly covered by Theorem~\ref{thm:mainThm}. Indeed, if \(d\geq 2\) is fixed, then for every \(\alpha>0\) one has
$d\leq \alpha\log N$ for large enough $N$. On the other hand, in that situation one can read off the sharper rate \(N^{-1/2}\) directly from the estimate \eqref{eq:RateProof} obtained at the end of the proof of Theorem \ref{thm:mainThm}. Indeed, since \(d\) is fixed, the prefactor appearing on the right-hand side there is just a constant depending only on \(d\). Thus, for every fixed \(d\geq 2\), there exists a constant \(C_d>0\) such that, for large enough $N$,
\[
d_K(W_N,\mathcal N(0,1))\leq \frac{C_d}{\sqrt N}.
\]
\end{remark}

We believe that Theorem~\ref{thm:mainThm} is interesting for several reasons. First, it provides a second-order refinement of the asymptotic formula \eqref{eq:expectation}. While \eqref{eq:expectation} identifies the deterministic first-order term \(1-e^{-1}\), the variance determines the scale of the fluctuations, and the Central Limit Theorem identifies their limiting distribution. In this sense, the theorem upgrades a statement about the typical size of \(V_N\) to a precise probabilistic description of the random fluctuations around its mean. 

Second, the result shows that the geometric dependence created by the overlapping caps is sufficiently weak to permit Gaussian asymptotics. This is far from obvious, since \(V_N\) is not a sum of independent contributions, but a highly non-additive geometric functional of a dependent union set. Theorem~\ref{thm:mainThm} shows that, despite these interactions, the cumulative fluctuation of the covered volume is still governed by a normal law.

Third, the theorem fits naturally into the balls-and-bins analogy introduced in \cite{HK-DCG}. In the classical occupancy problem, once the law of large numbers identifies the limiting proportion of occupied bins, the next natural question concerns the distribution of the fluctuations around that limit. Our result provides the corresponding answer in the geometric setting of random partial sphere coverings.

Fourth, the theorem complements the recent high-dimensional optimality result from \cite{HK-2026}. There it was shown that, under the mild assumption of antipodality, random partial coverings achieve the correct first-order asymptotic value \(1-e^{-1}\). Theorem~\ref{thm:mainThm} adds a probabilistic refinement to this picture by describing the shape and size of the random deviations around this optimal value.

Finally, the quantitative nature of our result is also of interest. The Kolmogorov bound not only implies convergence in distribution, but gives an explicit rate at which the law of the standardized covered-volume functional approaches the Gaussian distribution. Moreover, the argument is robust enough to remain valid beyond the fixed-dimensional setting and continues to apply when the dimension is allowed to grow at most logarithmically with \(N\).

\begin{remark}\label{rem:LEAN}
The proof of Theorem~\ref{thm:mainThm} has been formally verified in \textsc{Lean4} with the help of the \textsc{Aristotle} framework. Moreover, we have formally verified Proposition~\ref{prop:ShaoZhang} below, taken from \cite{ShaoZhang2025}, in a form that makes the constant explicit. More precisely, we formalized Corollary~2.5 of \cite{ShaoZhang2025} and obtained the optimal value \(C=12\) for the constant appearing there. The only ingredients in the proof of Theorem~\ref{thm:mainThm} that have not yet been formally verified are the lower variance bound in Proposition~\ref{thm:HK-variance} from \cite{HK-DCG} and Theorem~2.1 of \cite{ShaoZhang2025}, on which Corollary~2.5 relies. The corresponding \textsc{Lean4} files are available on the authors' websites.
\end{remark}

\section{Preliminaries}

To prove Theorem \ref{thm:mainThm}, we will need a few auxiliary results. The first comes from \cite{HK-DCG} and gives a variance bound for $V_N$ of which only the lower bound will actually be used in the sequel.

\begin{prop}\cite[Theorem 4.1]{HK-DCG}\label{thm:HK-variance}
    Let $C>0$ be an absolute constant and assume that $N\geq C^d$. Then
    \[
e^{-C_1}c_1^{d-1}N^{-1}\leq\var(V_N) \leq e^{-1}c_2^{d-1}N^{-1}
    \]
    where $c_1,c_2\in(0,1)$ and $C_1\in(1,2)$ are absolute constants.
\end{prop}

We also use the Berry--Esseen bound from \cite[Corollary~2.5]{ShaoZhang2025} for symmetric functions of independent random variables, see Remark \ref{rem:LEAN}. To rephrase it, let \(E\) be a measurable space, let \(X=(X_1,\dots,X_N)\) be a vector of i.i.d.\ \(E\)-valued random variables, let \(X'\) and \(\widetilde X\) be independent copies of \(X\), and let \(\mathcal C(X,X',\widetilde X)\) denote the collection of all recombinations of \(\{X,X',\widetilde X\}\). Here, by a \emph{recombination} we mean a random vector \(Z=(Z_1,\ldots,Z_N)\) such that, for each \(i\in[N]:=\{1,\ldots,N\}\), the coordinate \(Z_i\) is one of the three independent copies \(X_i\), \(X_i'\), and \(\widetilde X_i\), i.e., $Z_i\in\{X_i,X_i',\widetilde X_i\}$ for all $i\in[N]$. For \(Z\in \mathcal C(X,X',\widetilde X)\), \(i\in [N]\), and distinct \(i,j\in [N]\), define
\[
Z^{\{i\}}:=(Z_1,\dots,Z_{i-1},X_i',Z_{i+1},\dots,Z_N)
\]
and
\[
Z^{\{i,j\}}:=(Z_1,\dots,Z_{i-1},X_i',Z_{i+1},\dots,Z_{j-1},X_j',Z_{j+1},\dots,Z_N).
\]
If \(f:E^N\to\R\) is measurable, we set
\[
\Delta_i f(Z):=f(Z)-f(Z^{\{i\}})
\]
and
\begin{equation}\label{eq:Delta-i-j-def}
\Delta_{i,j}f(Z):=f(Z)-f(Z^{\{i\}})-f(Z^{\{j\}})+f(Z^{\{i,j\}}).
\end{equation}

\begin{prop}\label{prop:ShaoZhang}
Let $X=(X_1,\ldots,X_N)$ be a vector of $N\geq 1$ i.i.d.\ $E$-valued random variables and $f:E^N\to\R$ be a measurable function which is symmetric under all permutation of its arguments. Let $V_N:=f(X_1,\ldots,X_N)$, and define
\[
W_N:=\frac{V_N-\E[V_N]}{\sqrt{\var(V_N)}}.
\]
Then
\[
d_K(W_N,\mathcal N(0,1))
\leq
12\, N^{1/2}\var(V_N)^{-1}
\left(
\sqrt{N\delta_1}+N\sqrt{\delta_2}+\sqrt{\E[(\Delta_1 f(X))^4]}
\right),
\]
where
\[
\delta_1
:=
\sup_{Y,Z\in\mathcal C(X,X',\widetilde X)}
\E\!\left[
\mathbf 1\{\Delta_{1,2}f(Y)\neq 0\}\,(\Delta_1 f(Z))^4
\right]
\]
and
\[
\delta_2
:=
\sup_{Y,Y',Z\in\mathcal C(X,X',\widetilde X)}
\E\!\left[
\mathbf 1\{\Delta_{1,2}f(Y)\neq 0,\ \Delta_{1,3}f(Y')\neq 0\}\,(\Delta_2 f(Z))^4
\right].
\]
\end{prop}

In our situation, the function $f$ will be given by
\[
f(x_1,\dots,x_N):=\sigma\!\left(\bigcup_{i=1}^N C_N(x_i)\right),\qquad (x_1,\ldots,x_N)\in (\Sp)^N,
\]
so that the partial covered volume functional is \(V_N=f(X_1,\dots,X_N)\). Since \(f\) is symmetric, Proposition \ref{prop:ShaoZhang} applies. Thus, in order to prove Theorem \ref{thm:mainThm}, it will be enough to estimate the fourth moment of the first replacement difference and the interaction terms $\delta_1$ and $\delta_2$.

\section{Proof of Theorem \ref{thm:mainThm}}

Let
\[
f(x_1,\ldots,x_N):=\sigma\left(\bigcup_{i=1}^N C_N(x_i)\right),
\qquad (x_1,\ldots,x_N)\in (\Sp)^N,
\]
and recall that $\Delta_1 f(X)=f(X)-f(X^{\{1\}})$,
where \(X^{\{1\}}=(X_1',X_2,\ldots,X_N)\) and \(X_1'\) is an independent copy of \(X_1\).

\begin{lemma}\label{lem:first-difference-bound}
For every recombination $Z=(Z_1,\ldots,Z_N)\in\mathcal C(X,X',\widetilde X)$, we have
\[
|\Delta_1 f(Z)|\le \frac{1}{N}\qquad\text{a.s.}
\]
In particular,
\[
\E\big[(\Delta_1 f(Z))^4\big]\le \frac{1}{N^4}.
\]
\end{lemma}
\begin{proof}
Fix a recombination $Z=(Z_1,\ldots,Z_N)\in\mathcal C(X,X',\widetilde X)$ and write
\[
U:=\bigcup_{i=2}^N C_N(Z_i).
\]
Then
\[
f(Z)=\sigma\big(U\cup C_N(Z_1)\big)
\qquad\text{and}\qquad
f(Z^{\{1\}})=\sigma\big(U\cup C_N(X_1')\big).
\]
Therefore,
\[
\Delta_1 f(Z)
=
\sigma\big(U\cup C_N(Z_1)\big)-\sigma\big(U\cup C_N(X_1')\big).
\]
Using the identity
\[
\sigma(U\cup A)=\sigma(U)+\sigma(A\setminus U),
\]
valid for every measurable set \(A\subseteq \Sp\), we obtain
\[
\Delta_1 f(Z)
=
\sigma\big(C_N(Z_1)\setminus U\big)-\sigma\big(C_N(X_1')\setminus U\big).
\]
Now both quantities on the right-hand side are nonnegative and bounded above by \(1/N\), since each spherical cap has \(\sigma\)-measure \(1/N\). Hence
\[
|\Delta_1 f(Z)|\le \frac{1}{N}.
\]
Raising to the fourth power and taking expectations gives the second claim.
\end{proof}

To deal with the second-order replacement differences in $\delta_1$ and $\delta_2$, we will also need the following locality statement.

\begin{lemma}\label{lem:locality}
    Let $Y=(Y_1,\ldots,Y_N)\in\mathcal{C}(X,X',\widetilde{X})$, and set $U:=\bigcup_{i=3}^N C_N(Y_i)$ and
    \[
    A:=C_N(Y_1),\quad A':=C_N(X_1'),\quad B:=C_N(Y_2),\quad B':=C_N(X_2').
    \]
    Then
    \begin{align}\label{eq:locality-lem-eqn}
        \Delta_{1,2}f(Y) &= \sigma(U\cup A\cup B)-\sigma(U\cup A'\cup B)
        -\sigma(U\cup A\cup B')+\sigma(U\cup A'\cup B').
    \end{align}
    Moreover,
    \begin{align*}
\{\Delta_{1,2}f(Y)\neq 0\}\Longrightarrow\, &(A\cap B)\cup(A\cap B')\cup(A'\cap B)\cup(A'\cap B')\neq\varnothing.
    \end{align*}
\end{lemma}

\begin{proof}
    The identity for $\Delta_{1,2}f(Y)$ in the first part follows from \eqref{eq:Delta-i-j-def}. For the second part, assume that  
    \[
A\cap B=A\cap B'=A'\cap B=A'\cap B'=\varnothing.
    \]
We will show that $\Delta_{1,2}f(Y)=0$. If $A\cap B=\varnothing$, then 
\[
U\cup A\cup B=U\sqcup(A\setminus U)\sqcup (B\setminus U),
\]
where $\sqcup$ denotes the disjoint union. Hence,
\[
\sigma(U\cup A\cup B)=\sigma(U)+\sigma(A\setminus U)+\sigma(B\setminus U).
\]
Similarly, the conditions $A\cap B'=\varnothing$, $A'\cap B=\varnothing$, and $A'\cap B'=\varnothing$ imply
\begin{align*}
  \sigma(U\cup A'\cup B)&=\sigma(U)+\sigma(A'\setminus U)+\sigma(B\setminus U),\\
  \sigma(U\cup A\cup B')&=\sigma(U)+\sigma(A\setminus U)+\sigma(B'\setminus U),\\
  \sigma(U\cup A'\cup B')&=\sigma(U)+\sigma(A'\setminus U)+\sigma(B'\setminus U).
\end{align*}
Substituting these four expressions back into \eqref{eq:locality-lem-eqn}, all terms cancel and we get $\Delta_{1,2}f(Y)=0$. Taking the contrapositive statement, we see that the event $\Delta_{1,2}f(Y)\neq 0$ implies that at least one of the intersections $A\cap B, A\cap B', A'\cap B, A'\cap B'$ is nonempty.
\end{proof}

We now deal with the quantities $\delta_1$ and $\delta_2$ in Proposition~\ref{prop:ShaoZhang}. In a first step, we reduce them to an intersection probability for random spherical caps.

\begin{lemma}\label{lem:delta1-delta2-reduction}
Let $U_1,U_2$ be independent random points on $\Sp$ with distribution $\sigma$, and define
\[
    p_N := \PP\big(C_N(U_1)\cap C_N(U_2)\neq\varnothing\big).
\]
Then
\[
    \delta_1\le \frac{4}{N^4}\,p_N
    \qquad\text{and}\qquad
    \delta_2\le \frac{16}{N^4}\,p_N^2.
\]
\end{lemma}
\begin{proof}
By Lemma~\ref{lem:first-difference-bound}, for every recombination $Z\in\mathcal C(X,X',\widetilde X)$, we have $|\Delta_1 f(Z)|\le \frac{1}{N}$ almost surely. Hence,
\begin{equation}\label{eq:delta1-first-reduction}
\delta_1
\le
\frac{1}{N^4}
\sup_{Y\in\mathcal C(X,X',\widetilde X)}
\PP\big(\Delta_{1,2}f(Y)\neq 0\big).
\end{equation}
Now let $Y=(Y_1,\dots,Y_N)\in \mathcal C(X,X',\widetilde X)$. By Lemma \ref{lem:locality}, the event $\{\Delta_{1,2}f(Y)\neq 0\}$ can occur only if at least one of the following four intersections takes place:
\begin{align*}
C_N(Y_1)\cap C_N(Y_2)&\neq\varnothing,\\
C_N(Y_1)\cap C_N(X_2')&\neq\varnothing,\\
C_N(X_1')\cap C_N(Y_2)&\neq\varnothing,\\
C_N(X_1')\cap C_N(X_2')&\neq\varnothing.
\end{align*}
Therefore, by the union bound,
\begin{align}
\PP\big(\Delta_{1,2}f(Y)\neq 0\big)
&\le
\PP\big(C_N(Y_1)\cap C_N(Y_2)\neq\varnothing\big)\notag\\
&\quad+
\PP\big(C_N(Y_1)\cap C_N(X_2')\neq\varnothing\big)\notag\\
&\quad+
\PP\big(C_N(X_1')\cap C_N(Y_2)\neq\varnothing\big)\notag\\
&\quad+
\PP\big(C_N(X_1')\cap C_N(X_2')\neq\varnothing\big).\label{eq:union-bound-delta12}
\end{align}
Since in each of the four terms the two cap centers are independent and $\sigma$-distributed on $\Sp$, each probability on the right-hand side equals $p_N$. Thus,
\[
\PP\big(\Delta_{1,2}f(Y)\neq 0\big)\le 4p_N
\qquad\text{for all }Y\in\mathcal C(X,X',\widetilde X).
\]
Combining this with \eqref{eq:delta1-first-reduction}, we obtain
\[
\delta_1\le \frac{4}{N^4}\,p_N.
\]

We next consider the term $\delta_2$. Again by Lemma~\ref{lem:first-difference-bound}, now applied to the second coordinate, we have $|\Delta_2 f(Z)|\le \frac{1}{N}$ almost surely for every $Z\in\mathcal C(X,X',\widetilde X)$. Hence, as above,
\begin{equation}\label{eq:delta2-first-reduction}
\delta_2
\le
\frac{1}{N^4}
\sup_{Y,Y'\in\mathcal C(X,X',\widetilde X)}
\PP\big(\Delta_{1,2}f(Y)\neq 0,\ \Delta_{1,3}f(Y')\neq 0\big).
\end{equation}
Fix $Y,Y'\in\mathcal C(X,X',\widetilde X)$. As above, the event $\{\Delta_{1,2}f(Y)\neq 0\}$ can occur only if at least one of the four intersections
\begin{align}
C_N(Y_1)\cap C_N(Y_2)&\neq\varnothing,\notag\\
C_N(Y_1)\cap C_N(X_2')&\neq\varnothing,\notag\\
C_N(X_1')\cap C_N(Y_2)&\neq\varnothing,\notag\\
C_N(X_1')\cap C_N(X_2')&\neq\varnothing\label{eq:delta2-four-intersections-12}
\end{align}
occurs. Likewise, the event $\{\Delta_{1,3}f(Y')\neq 0\}$ can occur only if at least one of the four intersections
\begin{align}
C_N(Y_1')\cap C_N(Y_3')&\neq\varnothing,\notag\\
C_N(Y_1')\cap C_N(X_3')&\neq\varnothing,\notag\\
C_N(X_1')\cap C_N(Y_3')&\neq\varnothing,\notag\\
C_N(X_1')\cap C_N(X_3')&\neq\varnothing\label{eq:delta2-four-intersections-13}
\end{align}
occurs. Therefore, by another union bound,
\begin{equation}\label{eq:delta2-union}
\PP\big(\Delta_{1,2}f(Y)\neq 0,\ \Delta_{1,3}f(Y')\neq 0\big)
\le
\sum_{a=1}^4\sum_{b=1}^4 \PP(E_a\cap E_b'),
\end{equation}
where $E_1,\dots,E_4$ are the four events in \eqref{eq:delta2-four-intersections-12} and $E_1',\dots,E_4'$ are the four events in \eqref{eq:delta2-four-intersections-13}.

Fix $a,b\in\{1,\dots,4\}$. Then the event $E_a\cap E_b'$ can be written in the form
\[
E_a\cap E_b'
=
\{A\cap B\neq\varnothing,\ A'\cap C\neq\varnothing\},
\]
where $A,A'$ are caps determined by variables at coordinate $1$, $B$ is a cap determined by variables at coordinate $2$, and $C$ is a cap determined by variables at coordinate $3$. Since the variables at coordinates $2$ and $3$ are independent of each other and independent of those at coordinate $1$, the caps $B$ and $C$ are conditionally independent given $A,A'$. Consequently,
\[
\PP(E_a\cap E_b'\mid A,A')
=
\PP(A\cap B\neq\varnothing\mid A)\,
\PP(A'\cap C\neq\varnothing\mid A').
\]
Moreover, all caps of the form $C_N(u)$, $u\in\Sp$, have the same radius, and $\sigma$ is rotation-invariant on $\Sp$. Hence, if $U\in\Sp$ has distribution $\sigma$, then for every fixed cap $D$ of the form $C_N(u)$,
\[
\PP(C_N(U)\cap D\neq\varnothing)=p_N.
\]
Applying this with $D=A$ and $D=A'$, we obtain
\[
\PP(E_a\cap E_b'\mid A,A')=p_N^2.
\]
Taking expectations yields
\[
\PP(E_a\cap E_b')=p_N^2.
\]
Since this holds for every pair $(a,b)$, it follows from \eqref{eq:delta2-union} that
\[
\PP\big(\Delta_{1,2}f(Y)\neq 0,\ \Delta_{1,3}f(Y')\neq 0\big)
\le
\sum_{a=1}^4\sum_{b=1}^4 p_N^2
=
16\,p_N^2.
\]
Combining this with \eqref{eq:delta2-first-reduction}, we arrive at
\[
\delta_2\le \frac{16}{N^4}\,p_N^2.
\]
This completes the proof.
\end{proof}

In the next step, we bound the intersection probability $p_N$ for random spherical caps, which has appeared in Lemma \ref{lem:delta1-delta2-reduction}.

\begin{lemma}\label{lem:pN-bound}
For every \(N\ge 2\),
\[
p_N\le \frac{2^{d-1}}{N}.
\]
\end{lemma}

\begin{proof}
Let \(r_N\in(0,\pi)\) denote the geodesic radius of the cap \(C_N(x)\), $x\in\Sp$. Since \(\sigma(C_N(x))=1/N\) for every \(x\in\Sp\), the number \(r_N\) is well-defined and does not depend on \(x\).

We first note that
\[
p_N=\sigma(C_{2r_N}(x)),
\]
where \(C_{2r_N}(x)\) denotes the geodesic spherical cap of radius \(2r_N\) centered at \(x\). Indeed, two caps of radius \(r_N\) intersect if and only if the geodesic distance between their centers is at most \(2r_N\). Hence, if \(U_1,U_2\) are independent and \(\sigma\)-distributed on \(\Sp\), then, conditioning on \(U_1\) and using rotational invariance,
\[
p_N
=
\PP\big(C_N(U_1)\cap C_N(U_2)\neq\varnothing\big)
=
\PP\big(d_{\Sp}(U_1,U_2)\le 2r_N\big)
=
\sigma(C_{2r_N}(x))
\]
for any fixed \(x\in\Sp\), where $d_{\Sp}$ is the geodesic distance on $\Sp$.

Next, recall that the normalized surface measure of a geodesic spherical cap of radius \(r\in[0,\pi]\) is given by
\[
\sigma(C_r(x))
=
\frac{\int_0^r \sin^{d-2}(t)\,dt}{\int_0^\pi \sin^{d-2}(t)\,dt},
\]
independently of \(x\in\Sp\), see, e.g.,\ \cite[Equation (2.7)]{KabluchkoSteigenbergerThaele}. Since \(\sigma(C_{r_N}(x))=1/N\le 1/2\), we have \(r_N\le \pi/2\). Therefore,
\begin{align*}
p_N
&=
\frac{\int_0^{2r_N}\sin^{d-2}(t)\,dt}{\int_0^\pi \sin^{d-2}(t)\,dt} =\frac{2\int_0^{r_N}\sin^{d-2}(2u)\,du}{\int_0^\pi \sin^{d-2}(t)\,dt}
\end{align*}
by the substitution \(t=2u\). Using that
\[
\sin(2u)=2\sin(u)\cos(u)\le 2\sin(u),
\qquad u\in[0,\pi/2],
\]
we obtain
\[
\sin^{d-2}(2u)\le 2^{d-2}\sin^{d-2}(u),
\qquad u\in[0,r_N].
\]
Consequently,
\begin{align*}
p_N
&\le
\frac{2^{d-1}\int_0^{r_N}\sin^{d-2}(u)\,du}{\int_0^\pi \sin^{d-2}(t)\,dt}=
2^{d-1}\sigma(C_{r_N}(x))
=
\frac{2^{d-1}}{N}.
\end{align*}
This proves the claim.
\end{proof}

We are now in a position to derive the Kolmogorov bound for the standardized covered-volume functional
\[
W_N:=\frac{V_N-\E[V_N]}{\sqrt{\var(V_N)}}.
\]

\begin{proof}[Proof of Theorem \ref{thm:mainThm}]
Substituting the bounds of Lemma~\ref{lem:first-difference-bound} and Lemma~\ref{lem:delta1-delta2-reduction} into Proposition~\ref{prop:ShaoZhang} yields
\begin{align*}
d_K(W_N,\mathcal N(0,1))
&\le
12\, N^{1/2}\var(V_N)^{-1}
\left(
\sqrt{\frac{4p_N}{N^3}}
+
N\sqrt{\frac{16p_N^2}{N^4}}
+
\sqrt{\frac{16}{N^4}}
\right)\\
&=
12\,\var(V_N)^{-1}
\left(
\frac{2\sqrt{p_N}}{N}
+
\frac{4\,p_N}{\sqrt N}
+
\frac{4}{N^{3/2}}
\right).
\end{align*}
Next, we use Lemma~\ref{lem:pN-bound} to upper bound the intersection probability $p_N$ and Proposition~\ref{thm:HK-variance} to lower bound the variance $\var(V_N)$. This yields
\[
d_K(W_N,\mathcal N(0,1))
\le
12\,e^{C_1}c_1^{1-d}
\left(2^{(d+1)/2}+2^{d+1}+4\right)\frac1{\sqrt N}.
\]
Next, we note that, for \(d\geq 2\), $2^{(d+1)/2}\leq 2^{d+1}$ and  $4\leq 2^{d+1}$,
and therefore $2^{(d+1)/2}+2^{d+1}+4\leq 3\cdot 2^{d+1}$.
Consequently,
\begin{equation}\label{eq:RateProof}
d_K(W_N,\mathcal N(0,1))
\le
36\,e^{C_1}c_1^{1-d}2^{d+1}\frac1{\sqrt N}
=
72\,e^{C_1}c_1\,
\frac{(2/c_1)^d}{\sqrt N}.    
\end{equation}
Now assume that $d\leq \alpha \log N$ for some \(\alpha>0\). Then
\[
\Big(\frac{2}{c_1}\Big)^d
\leq
\Big(\frac{2}{c_1}\Big)^{\alpha\log N}
=
N^{\alpha\log(2/c_1)},
\]
so that
\[
d_K(W_N,\mathcal N(0,1))
\le
72\,e^{C_1}c_1\,
N^{\alpha\log(2/c_1)-1/2},
\]
which is the desired error bound. Now, by assumption, $\alpha<\frac{1}{2\log(2/c_1)}$, so we have $\alpha\log(2/c_1)-\frac12<0$ and hence
\[
N^{\alpha\log(2/c_1)-1/2}\longrightarrow 0
\qquad\text{as }N\to\infty.
\]
It follows that
\[
d_K(W_N,\mathcal N(0,1))\longrightarrow 0
\qquad\text{as }N\to\infty.
\]
Finally, since convergence in the Kolmogorov distance implies convergence in distribution, we conclude that
\[
W_N\stackrel{d}{\longrightarrow}\mathcal{N}(0,1)\qquad\text{as }N\to\infty.
\]
This proves the Central Limit Theorem.
\end{proof}

%%%%%%%%%%%%%%%%%%%%%%%%%%%
\section*{Acknowledgments}

We would like to thank Philipp Tuchel (Bochum) for interesting conversations about formal verifications in \textsc{Lean}. CT was supported by the DFG via SPP 2265 \textit{Random Geometric Systems}.

%%%%%%%%%%%%%%%%%%%%%%%%%%%%%%%%%%%%%%
\bibliographystyle{plain}
\bibliography{main}

%%%%%%%%%%%%%%

\vspace{3mm}

\noindent {\sc Department of Mathematics \& Computer Science, Longwood University, 201 High St., Farmville, Virginia 23901, USA}

\noindent {\it E-mail address:} {\tt hoehnersd@longwood.edu}

\vspace{3mm}

\noindent {\sc Faculty of Mathematics, Ruhr University Bochum, 44801 Bochum, Germany}

\noindent {\it E-mail address:} {\tt christoph.thaele@rub.de}

\end{document}